\documentclass[11pt,a4paper]{article}

\usepackage{authblk}

\usepackage[T1]{fontenc}
\usepackage[utf8]{inputenc}
\usepackage{newcent}
\usepackage{titlesec}
\usepackage{enumerate}
\usepackage{changes}
\usepackage{hyperref}
\usepackage{amsmath}
\usepackage{amsfonts}
\usepackage{amssymb}
\usepackage{amsthm}
\usepackage{newlfont}
\usepackage{mathtools}
\usepackage[top=1.5in, bottom=1.5in, left=1.2in, right=1.2in]{geometry}
\usepackage{tikz}

\theoremstyle{plain}
\newtheorem{theorem}{Theorem}[section]
\newtheorem{proposition}[theorem]{Proposition}

\newtheorem{lemma}[theorem]{Lemma}

\newtheorem{assumption}[theorem]{Assumption}

\theoremstyle{definition}
\newtheorem{definition}[theorem]{Definition}

\newtheorem*{remark}{Remark}

\newcommand{\RR}{\mathbb{R}}
\newcommand{\NN}{\mathbb{N}}
\newcommand{\ZZ}{\mathbb{Z}}

\newcommand{\cC}{\mathcal{C}}
\newcommand{\drm}{\mathrm{d}}
\newcommand{\euler}{\mathrm{e}}

\DeclareMathOperator{\supp}{supp}

\DeclareMathOperator{\Deg}{Deg}

\DeclareMathOperator{\spec}{spec}

\DeclareMathOperator{\arsinh}{arsinh}

\DeclareMathOperator{\Eins}{\mathbf{1}}

\newcommand{\eat}[1]{}

\let\oldint\int
\renewcommand{\int}{\oldint\limits}

\newcommand{\Hmm}[1]{\leavevmode{\marginpar{\tiny%
			$\hbox to 0mm{\hspace*{-0.5mm}$\leftarrow$\hss}%
			\vcenter{\vrule depth 0.1mm height 0.1mm width \the\marginparwidth}%
			\hbox to
			0mm{\hss$\rightarrow$\hspace*{-0.5mm}}$\\\relax\raggedright #1}}}

\numberwithin{equation}{section}

\begin{document}

\title{Anchored heat kernel upper bounds on graphs with unbounded geometry and anti-trees}
\author{Matthias Keller\thanks{matthias.keller@uni-potsdam.de} }
\author{Christian Rose\thanks{christian.rose@uni-potsdam.de}}
\affil[]{Institut f\"ur Mathematik, Universit\"at Potsdam, 		14476  Potsdam, Germany}
\maketitle

\begin{abstract}
We derive Gaussian heat kernel bounds on graphs with respect to a fixed origin for large times
under the assumption of a Sobolev inequality and volume doubling on large balls.
The upper bound from our previous work \cite{KellerRose-22a} is affected by a new correction term measuring the distance to the origin. The main result is then applied to anti-trees with unbounded vertex degree, yielding Gaussian upper bounds for this class of graphs for the first time. In order to prove this, we show that isoperimetric estimates with respect to intrinsic metrics yield Sobolev inequalities. Finally, we prove that anti-trees are Ahlfors regular and that they satisfy an isoperimetric inequality of a larger dimension. 
\\
	
\end{abstract}

The focus of the present note lies on Gaussian upper bounds for heat kernels on graphs with possibly unbounded geometry. Seminal works go back to Davies and Delmotte for graphs with bounded geometry, \cite{Davies-93a, Delmotte-99}. 
In contrast, we will be dealing with graphs which do not satisfy any a priori boundedness assumption on the local geometry of the graph.
First results in this direction have been obtained in \cite{Folz-11, AndresDS-16, BarlowChen-16, BauerHuaYau-17,ChenKW-20b}.
Recently, the authors obtained  Gaussian bounds for heat kernels for large times on graphs with unbounded geometry. The assumptions include Sobolev inequalities and volume doubling  in large balls in terms of intrinsic metrics 
\cite{KellerRose-22a}.
\\
Here, we derive an anchored version of these recent Gaussian upper bounds of the heat kernel $p$ of locally finite weighted discrete graphs over a discrete measure space $(X,m)$, cf.~Section~\ref{section:anchoredgaussian}. Applied to anti-trees, our results yield for the first time Gaussian upper bounds for their heat kernels and, hence, provide explicit examples of graphs with unbounded vertex degree.
\\
In order to highlight our contribution it is instructive to recall a simplified version of the main result Theorem~6.1 from \cite{KellerRose-22a}. Assume that the graph is equipped with an intrinsic metric $\rho$ of jump size one. Moreover, all large balls around $x,y\in X$ are supposed to satisfy a Sobolev inequality with exponent $n>2$ and volume doubling. 
Then, the authors obtained  Gaussian upper bounds of the form
\[
p_t(x,y)
\leq 
f(x,y,t)\frac{\left(1\vee ( \sqrt{t^2+\rho(x,y)^2}-t)\right)^{\frac{n}{2}}}{m(B_x(\sqrt t))m(B_y(\sqrt t))}\euler^{-\zeta_1(\rho(x,y),t)}
\]
for large $t>0$.
The function $\zeta_1$ is optimal as observed by Davies and satisfies 
\[
\zeta_1(r,t)\sim \frac{r^2}{2t}, \quad t\to\infty, \ r>0,
\]
meaning that the fraction of the left- and right-hand side converges to one. 
Moreover, the polynomial correction term has different behavior as $t\to 0$ than the corresponding term on manifolds, cf.~\cite{Grigoryan-09}. The function $f$ measures the unboundedness of the geometry about $x$ and $y$. It is bounded if, e.g., $m$ is the normalizing measure, or the $L^p$-means of vertex degree and inverted measure in balls grow at most exponentially. Note that no a priori boundedness of the vertex degree or the measure is assumed. 
\\

The aim of this note is twofold. 
\\
First, we will relax the conditions on the vertices. Instead of assuming that large balls around the vertices $x$ and $y$ fulfill Sobolev and volume doubling, we rather impose these properties  about a root vertex $o\in X$. This yields Gaussian upper bounds for all vertices with an additional polynomial correction term measuring the distance to $o$. In fact, Theorem~\ref{thm:main1} states that for all $x,y\in X$ and $t> 0$ large
\begin{equation*}
p_t(x,y)
\leq 
f_o(x,y,t)\left(1+\frac{\rho(o,x)^2+\rho(o,y)^2}{t}\right)^{\frac{n}{2}}
\frac{\left(1\vee ( \sqrt{t^2+\rho(x,y)^2}-t)\right)^{\frac{n}{2}}}{m(B_o(\sqrt t))}\euler^{-\zeta_1(\rho(x,y),t)}.
\end{equation*}
The function $f_o$ measures the unboundedness of the geometry with respect to  $o$, incorporating the distance from $x$ and $y$. The function $f_o$ is bounded if $m$ is the normalizing measure or weighted $L^p$-means of the vertex degree and inverted measure about $o$ grow at most exponentially.  
\\
The above statements require Sobolev and volume doubling for all large balls. Our main result Theorem~\ref{thm:main1} even yields a localized version involving the spectral bottom of the Laplacian. To the best of our knowledge, this is the first variant of this result in this form on graphs.
\\
For degenerating parabolic equations  in $\RR^n$, Zhikov imposed conditions at the origin and derived off-diagonal Gaussian upper bounds, \cite{Zhikov-13}. A similar approach was used for elliptic operators on $\ZZ^n$ with unbounded but integrable weights in \cite{MourratOtto-16, AndresDS-16}. However, the latter articles do not provide the optimal Gaussian estimate.
\\

\begin{figure}[ht]
  \begin{minipage}[b]{\linewidth}
  \centering
\begin{tikzpicture}[scale=.6, yscale=.7]
  dot/.style={fill=blue,circle,minimum size=2pt}]
\tikzstyle{every node}=[draw, shape=circle, fill=black, inner sep=.5pt]

 \node at (0, 0)   (a) {};

    \node at (3, 2)   (b1) {};
    \node at (3, 0)  (b2)     {};
    \node at (3, -2)  (b3)     {};

\node at (6, 3)   (c1) {};
    \node at (6,2)  (c2)     {};
    \node at (6, 1)  (c3)     {};
    \node at (6, 0)   (c4) {};
    \node at (6, -1)  (c5)     {};
    \node at (6, -2)  (c6)     {};
    \node at (6,-3) (c7) {};

\node at (11, 6)   (d1) {};
    \node at (11,5)  (d2)     {};
    \node at (11, 4)  (d3)     {};
    \node at (11, 3)   (d4) {};
    \node at (11, 2)  (d5)     {};
    \node at (11, 1)  (d6)     {};
    \node at (11,0) (d7) {};
    \node at (11, -1)   (d8) {};
    \node at (11,-2)  (d9)     {};
    \node at (11, -3)  (d10)     {};
    \node at (11, -4)   (d11) {};
    \node at (11, -5)  (d12)     {};
    \node at (11, -6)  (d13)     {};

\node at (19, 9)   (e1) {};
\node at (19, 8)   (e2) {};
\node at (19, 7)   (e3) {};
\node at (19, 6)   (e4) {};
    \node at (19,5)  (e5)     {};
    \node at (19, 4)  (e6)     {};
    \node at (19, 3)   (e7) {};
    \node at (19, 2)  (e8)     {};
    \node at (19, 1)  (e9)     {};
    \node at (19,0) (e10) {};
    \node at (19, -1)   (e11) {};
    \node at (19,-2)  (e12)     {};
    \node at (19, -3)  (e13)     {};
    \node at (19, -4)   (e14) {};
    \node at (19, -5)  (e15)     {};
    \node at (19, -6)  (e16)     {};
    \node at (19, -7)   (e17) {};
\node at (19, -8)   (e18) {};
\node at (19, -9)   (e19) {};

\node at (21,0) (f15) {};
\node at (21.4,0) (f16) {};
\node at (21.8,0) (f17) {};


    \draw (a) -- (b1)   {};
 
   \draw(a) -- (b2)   {};
 
\draw (a)-- (b3) {};

\draw (b1) -- (c1) {};

\draw (b1) -- (c2) {};

\draw (b1) -- (c3) {};

\draw (b1) -- (c4) {};

\draw (b1) -- (c5) {};

\draw (b1) -- (c6) {};

\draw (b1) -- (c7) {};

\draw (b2) -- (c1) {};
\draw (b2) -- (c2) {};
\draw (b2) -- (c3) {};
\draw (b2) -- (c4) {};
\draw (b2) -- (c5) {};
\draw (b2) -- (c6) {};
\draw (b2) -- (c7) {};

\draw (b3) -- (c1) {};
\draw (b3) -- (c2) {};
\draw (b3) -- (c3) {};
\draw (b3) -- (c4) {};
\draw (b3) -- (c5) {};
\draw (b3) -- (c6) {};
\draw (b3) -- (c7) {};

\draw (c1) -- (d1) {};
\draw (c1) -- (d2) {};
\draw (c1) -- (d3) {};
\draw (c1) -- (d4) {};
\draw (c1) -- (d5) {};
\draw (c1) -- (d6) {};
\draw (c1) -- (d7) {};
\draw (c1) -- (d8) {};
\draw (c1) -- (d9) {};
\draw (c1) -- (d10) {};
\draw (c1) -- (d11) {};
\draw (c1) -- (d12) {};
\draw (c1) -- (d13) {};

\draw (c2) -- (d1) {};
\draw (c2) -- (d2) {};
\draw (c2) -- (d3) {};
\draw (c2) -- (d4) {};
\draw (c2) -- (d5) {};
\draw (c2) -- (d6) {};
\draw (c2) -- (d7) {};
\draw (c2) -- (d8) {};
\draw (c2) -- (d9) {};
\draw (c2) -- (d10) {};
\draw (c2) -- (d11) {};
\draw (c2) -- (d12) {};
\draw (c2) -- (d13) {};

\draw (c3) -- (d1) {};
\draw (c3) -- (d2) {};
\draw (c3) -- (d3) {};
\draw (c3) -- (d4) {};
\draw (c3) -- (d5) {};
\draw (c3) -- (d6) {};
\draw (c3) -- (d7) {};
\draw (c3) -- (d8) {};
\draw (c3) -- (d9) {};
\draw (c3) -- (d10) {};
\draw (c3) -- (d11) {};
\draw (c3) -- (d12) {};
\draw (c3) -- (d13) {};

\draw (c4) -- (d1) {};
\draw (c4) -- (d2) {};
\draw (c4) -- (d3) {};
\draw (c4) -- (d4) {};
\draw (c4) -- (d5) {};
\draw (c4) -- (d6) {};
\draw (c4) -- (d7) {};
\draw (c4) -- (d8) {};
\draw (c4) -- (d9) {};
\draw (c4) -- (d10) {};
\draw (c4) -- (d11) {};
\draw (c4) -- (d12) {};
\draw (c4) -- (d13) {};

\draw (c5) -- (d1) {};
\draw (c5) -- (d2) {};
\draw (c5) -- (d3) {};
\draw (c5) -- (d4) {};
\draw (c5) -- (d5) {};
\draw (c5) -- (d6) {};
\draw (c5) -- (d7) {};
\draw (c5) -- (d8) {};
\draw (c5) -- (d9) {};
\draw (c5) -- (d10) {};
\draw (c5) -- (d11) {};
\draw (c5) -- (d12) {};
\draw (c5) -- (d13) {};

\draw (c6) -- (d1) {};
\draw (c6) -- (d2) {};
\draw (c6) -- (d3) {};
\draw (c6) -- (d4) {};
\draw (c6) -- (d5) {};
\draw (c6) -- (d6) {};
\draw (c6) -- (d7) {};
\draw (c6) -- (d8) {};
\draw (c6) -- (d9) {};
\draw (c6) -- (d10) {};
\draw (c6) -- (d11) {};
\draw (c6) -- (d12) {};
\draw (c6) -- (d13) {};

\draw (c7) -- (d1) {};
\draw (c7) -- (d2) {};
\draw (c7) -- (d3) {};
\draw (c7) -- (d4) {};
\draw (c7) -- (d5) {};
\draw (c7) -- (d6) {};
\draw (c7) -- (d7) {};
\draw (c7) -- (d8) {};
\draw (c7) -- (d9) {};
\draw (c7) -- (d10) {};
\draw (c7) -- (d11) {};
\draw (c7) -- (d12) {};
\draw (c7) -- (d13) {};

\draw (d1) -- (e1) {};
\draw (d1) -- (e2) {};
\draw (d1) -- (e3) {};
\draw (d1) -- (e4) {};
\draw (d1) -- (e5) {};
\draw (d1) -- (e6) {};
\draw (d1) -- (e7) {};
\draw (d1) -- (e8) {};
\draw (d1) -- (e9) {};
\draw (d1) -- (e10) {};
\draw (d1) -- (e11) {};
\draw (d1) -- (e12) {};
\draw (d1) -- (e13) {};
\draw (d1) -- (e14) {};
\draw (d1) -- (e15) {};
\draw (d1) -- (e16) {};
\draw (d1) -- (e17) {};
\draw (d1) -- (e18) {};
\draw (d1) -- (e19) {};

\draw (d2) -- (e1) {};
\draw (d2) -- (e2) {};
\draw (d2) -- (e3) {};
\draw (d2) -- (e4) {};
\draw (d2) -- (e5) {};
\draw (d2) -- (e6) {};
\draw (d2) -- (e7) {};
\draw (d2) -- (e8) {};
\draw (d2) -- (e9) {};
\draw (d2) -- (e10) {};
\draw (d2) -- (e11) {};
\draw (d2) -- (e12) {};
\draw (d2) -- (e13) {};
\draw (d2) -- (e14) {};
\draw (d2) -- (e15) {};
\draw (d2) -- (e16) {};
\draw (d2) -- (e17) {};
\draw (d2) -- (e18) {};
\draw (d2) -- (e19) {};

\draw (d3) -- (e1) {};
\draw (d3) -- (e2) {};
\draw (d3) -- (e3) {};
\draw (d3) -- (e4) {};
\draw (d3) -- (e5) {};
\draw (d3) -- (e6) {};
\draw (d3) -- (e7) {};
\draw (d3) -- (e8) {};
\draw (d3) -- (e9) {};
\draw (d3) -- (e10) {};
\draw (d3) -- (e11) {};
\draw (d3) -- (e12) {};
\draw (d3) -- (e13) {};
\draw (d3) -- (e14) {};
\draw (d3) -- (e15) {};
\draw (d3) -- (e16) {};
\draw (d3) -- (e17) {};
\draw (d3) -- (e18) {};
\draw (d3) -- (e19) {};

\draw (d4) -- (e1) {};
\draw (d4) -- (e2) {};
\draw (d4) -- (e3) {};
\draw (d4) -- (e4) {};
\draw (d4) -- (e5) {};
\draw (d4) -- (e6) {};
\draw (d4) -- (e7) {};
\draw (d4) -- (e8) {};
\draw (d4) -- (e9) {};
\draw (d4) -- (e10) {};
\draw (d4) -- (e11) {};
\draw (d4) -- (e12) {};
\draw (d4) -- (e13) {};
\draw (d4) -- (e14) {};
\draw (d4) -- (e15) {};
\draw (d4) -- (e16) {};
\draw (d4) -- (e17) {};
\draw (d4) -- (e18) {};
\draw (d4) -- (e19) {};

\draw (d5) -- (e1) {};
\draw (d5) -- (e2) {};
\draw (d5) -- (e3) {};
\draw (d5) -- (e4) {};
\draw (d5) -- (e5) {};
\draw (d5) -- (e6) {};
\draw (d5) -- (e7) {};
\draw (d5) -- (e8) {};
\draw (d5) -- (e9) {};
\draw (d5) -- (e10) {};
\draw (d5) -- (e11) {};
\draw (d5) -- (e12) {};
\draw (d5) -- (e13) {};
\draw (d5) -- (e14) {};
\draw (d5) -- (e15) {};
\draw (d5) -- (e16) {};
\draw (d5) -- (e17) {};
\draw (d5) -- (e18) {};
\draw (d5) -- (e19) {};

\draw (d6) -- (e1) {};
\draw (d6) -- (e2) {};
\draw (d6) -- (e3) {};
\draw (d6) -- (e4) {};
\draw (d6) -- (e5) {};
\draw (d6) -- (e6) {};
\draw (d6) -- (e7) {};
\draw (d6) -- (e8) {};
\draw (d6) -- (e9) {};
\draw (d6) -- (e10) {};
\draw (d6) -- (e11) {};
\draw (d6) -- (e12) {};
\draw (d6) -- (e13) {};
\draw (d6) -- (e14) {};
\draw (d6) -- (e15) {};
\draw (d6) -- (e16) {};
\draw (d6) -- (e17) {};
\draw (d6) -- (e18) {};
\draw (d6) -- (e19) {};

\draw (d7) -- (e1) {};
\draw (d7) -- (e2) {};
\draw (d7) -- (e3) {};
\draw (d7) -- (e4) {};
\draw (d7) -- (e5) {};
\draw (d7) -- (e6) {};
\draw (d7) -- (e7) {};
\draw (d7) -- (e8) {};
\draw (d7) -- (e9) {};
\draw (d7) -- (e10) {};
\draw (d7) -- (e11) {};
\draw (d7) -- (e12) {};
\draw (d7) -- (e13) {};
\draw (d7) -- (e14) {};
\draw (d7) -- (e15) {};
\draw (d7) -- (e16) {};
\draw (d7) -- (e17) {};
\draw (d7) -- (e18) {};
\draw (d7) -- (e19) {};

\draw (d8) -- (e1) {};
\draw (d8) -- (e2) {};
\draw (d8) -- (e3) {};
\draw (d8) -- (e4) {};
\draw (d8) -- (e5) {};
\draw (d8) -- (e6) {};
\draw (d8) -- (e7) {};
\draw (d8) -- (e8) {};
\draw (d8) -- (e9) {};
\draw (d8) -- (e10) {};
\draw (d8) -- (e11) {};
\draw (d8) -- (e12) {};
\draw (d8) -- (e13) {};
\draw (d8) -- (e14) {};
\draw (d8) -- (e15) {};
\draw (d8) -- (e16) {};
\draw (d8) -- (e17) {};
\draw (d8) -- (e18) {};
\draw (d8) -- (e19) {};

\draw (d9) -- (e1) {};
\draw (d9) -- (e2) {};
\draw (d9) -- (e3) {};
\draw (d9) -- (e4) {};
\draw (d9) -- (e5) {};
\draw (d9) -- (e6) {};
\draw (d9) -- (e7) {};
\draw (d9) -- (e8) {};
\draw (d9) -- (e9) {};
\draw (d9) -- (e10) {};
\draw (d9) -- (e11) {};
\draw (d9) -- (e12) {};
\draw (d9) -- (e13) {};
\draw (d9) -- (e14) {};
\draw (d9) -- (e15) {};
\draw (d9) -- (e16) {};
\draw (d9) -- (e17) {};
\draw (d9) -- (e18) {};
\draw (d9) -- (e19) {};

\draw (d10) -- (e1) {};
\draw (d10) -- (e2) {};
\draw (d10) -- (e3) {};
\draw (d10) -- (e4) {};
\draw (d10) -- (e5) {};
\draw (d10) -- (e6) {};
\draw (d10) -- (e7) {};
\draw (d10) -- (e8) {};
\draw (d10) -- (e9) {};
\draw (d10) -- (e10) {};
\draw (d10) -- (e11) {};
\draw (d10) -- (e12) {};
\draw (d10) -- (e13) {};
\draw (d10) -- (e14) {};
\draw (d10) -- (e15) {};
\draw (d10) -- (e16) {};
\draw (d10) -- (e17) {};
\draw (d10) -- (e18) {};
\draw (d10) -- (e19) {};

\draw (d11) -- (e1) {};
\draw (d11) -- (e2) {};
\draw (d11) -- (e3) {};
\draw (d11) -- (e4) {};
\draw (d11) -- (e5) {};
\draw (d11) -- (e6) {};
\draw (d11) -- (e7) {};
\draw (d11) -- (e8) {};
\draw (d11) -- (e9) {};
\draw (d11) -- (e10) {};
\draw (d11) -- (e11) {};
\draw (d11) -- (e12) {};
\draw (d11) -- (e13) {};
\draw (d11) -- (e14) {};
\draw (d11) -- (e15) {};
\draw (d11) -- (e16) {};
\draw (d11) -- (e17) {};
\draw (d11) -- (e18) {};
\draw (d11) -- (e19) {};

\draw (d12) -- (e1) {};
\draw (d12) -- (e2) {};
\draw (d12) -- (e3) {};
\draw (d12) -- (e4) {};
\draw (d12) -- (e5) {};
\draw (d12) -- (e6) {};
\draw (d12) -- (e7) {};
\draw (d12) -- (e8) {};
\draw (d12) -- (e9) {};
\draw (d12) -- (e10) {};
\draw (d12) -- (e11) {};
\draw (d12) -- (e12) {};
\draw (d12) -- (e13) {};
\draw (d12) -- (e14) {};
\draw (d12) -- (e15) {};
\draw (d12) -- (e16) {};
\draw (d12) -- (e17) {};
\draw (d12) -- (e18) {};
\draw (d12) -- (e19) {};

\draw (d13) -- (e1) {};
\draw (d13) -- (e2) {};
\draw (d13) -- (e3) {};
\draw (d13) -- (e4) {};
\draw (d13) -- (e5) {};
\draw (d13) -- (e6) {};
\draw (d13) -- (e7) {};
\draw (d13) -- (e8) {};
\draw (d13) -- (e9) {};
\draw (d13) -- (e10) {};
\draw (d13) -- (e11) {};
\draw (d13) -- (e12) {};
\draw (d13) -- (e13) {};
\draw (d13) -- (e14) {};
\draw (d13) -- (e15) {};
\draw (d13) -- (e16) {};
\draw (d13) -- (e17) {};
\draw (d13) -- (e18) {};
\draw (d13) -- (e19) {};

\end{tikzpicture}

\caption{First spheres of an anti-tree.}
\end{minipage}
\end{figure}
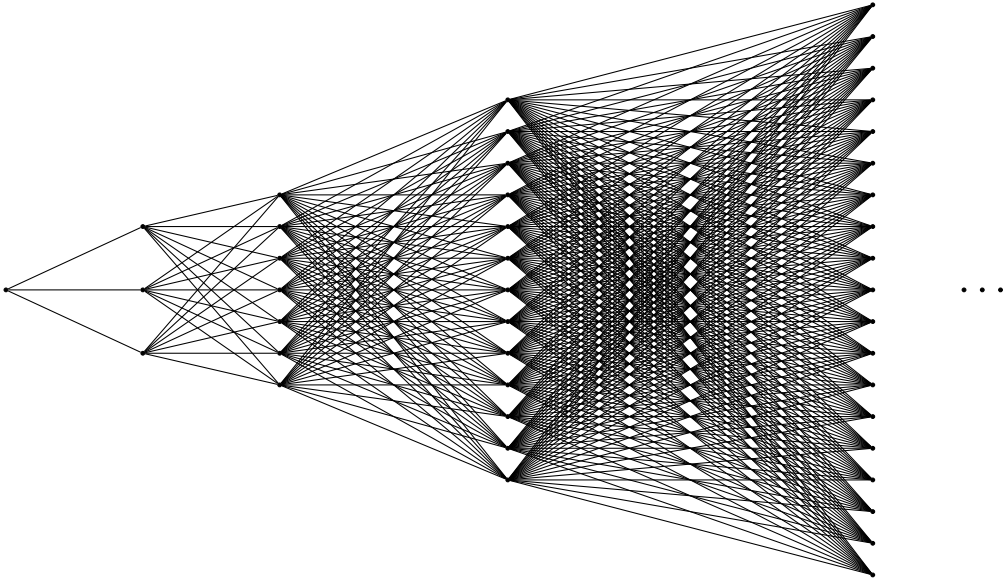

Our second aim concerns the application of our results to anti-trees. These graphs play a prominent role in the theory of graphs with unbounded geometry as they provide examples for similarities and disparities between graphs and manifolds.
Specifically, considering the combinatorial graph distance, they includeexamples of polynomial growth which are stochastically incomplete \cite{Wojciechowski11}  and have a spectral gap \cite{KellerLW-13}. This disparity was later resolved by the use of intrinsic metrics \cite{Folz-14,HaeselerKW-13, Huang, HuangKS-20, KellerLW-21}.

To illustrate the anti-trees we are considering here, let $\gamma\geq 0$, $X$ an anti-tree with root $o$ an which has $s_{k-1}=[ k^\gamma]$ vertices in the sphere with combinatorial distance $ k \ge1$ to $ o $. In contrast to trees where each vertex has only one backwards neighbor, anti-trees have complete bipartite graphs between spheres. Furthermore, let $\rho$ the intrinsic path degree metric, cf.~Section~\ref{section:antitrees}. Then if $\gamma=0$, the anti-tree is isomorphic to $\NN$, for $\gamma\in(0,2)$ it has polynomial volume growth, for $\gamma=2$ the volume growth is exponential, and for $\gamma>2$ the intrinsic diameter is bounded (with respect to the intrinsic metric $ \rho $).

So far, there were no results on heat kernel bounds of the above form for anti-trees as their vertex degree is not bounded.  For $\gamma\in(0,2)$, we denote $d=\frac{2(\gamma+1)}{2-\gamma}  $ which appears as the volume growth dimension. Then there is
$C_0>0$  such that  for all $x,y\in X$ with either $ \rho(o,x)\neq \rho(o,y) $ or $ x=y $ and  $t\geq 2\cdot 72^2$ we have
	
\[		p_{t}(x,y)
		\leq 
	 C 
	 \left(1+\frac{\rho(o,x)^2+\rho(o,y)^2}{t}\right)^{d}
		\cdot \frac{ \left(1\vee\left(\sqrt{t^2+\rho(x,y)^2}-t\right)\right)^{d}}{m(B_{o}(\sqrt {t}))}
		\euler^{-\zeta_1\left(\rho(x,y),t\right)}.
		\]
Replacing the intrinsic metric $ \rho $ with the combinatorial graph metric, we denote the combinatorial distance of $ x $ to the root $ o $ by $ |x| $. Then, we can further estimate this heat kernel bound to obtain the following for large $t>0$
\begin{align*}
		p_{t}(x,y)
	\leq 
	C 
	\left(1+\frac{ |x|^{2({\gamma+1})}+|y|^{2(\gamma+1)}}{t^{d}}\right)
	\cdot \frac{ 1}{t^{\frac{d}{2}}}
		\euler^{-\frac{ ||x|-|y||^{2-\gamma}}{Ct}}.
\end{align*}

In the proof we show that such anti-trees satisfiy certain  isoperimetric estimates in all distance balls which then in turn implies the required Sobolev inequalities. Albeit the latter implication is classical for manifolds, we could not find a reference  on general graphs with respect to intrinsic metrics. Therefore, we generalize the proof of the Cheeger inequality to our setting from \cite{KellerLW-21}.
\\
Similar Gaussian bounds for the heat kernel on anti-trees could be derived from the main results in \cite{KellerRose-22a} in conjunction with our developments on isoperimetric constants. However, this bound would not include the first polynomial correction term displayed above and would only hold for larger times, cf.~Theorem~\ref{thm:antitree2}.
\\

The outline of this paper is as follows. In Section~\ref{section:mainresults}, we introduce weighted graphs, their Laplacian, intrinsic metric, and the heat kernel. Moreover, we recall the Sobolev and volume doubling properties we are working with. 
In Section~\ref{section:anchoredgaussian} we formulate and prove our main result. The implication that isoperimetric estimates yield Sobolev inequalities is proven in Section~\ref{section:iso}.
Section~\ref{section:antitrees} is devoted to the study of anti-trees. We show that their heat kernels are spherically symmetric and that their study can therefore  be reduced to the study of the heat kernel on a particular model graph on $\NN_0$. It will be shown that these models  satisfy isoperimetry and Ahlfors regularity, i.e., polynomial volume growth, if the sphere function is chosen as above.

\section{The set-up}\label{section:mainresults}

Let $X$ be a countable set, $m\colon X\to(0,\infty)$ a measure on $X$, and $b$ a locally finite connected graph over $(X,m)$.  We denote by $\cC(X)$ the set of continuous functions on $X$, and $\nabla_{xy}f:=f(x)-f(y)$, $f\in\cC(X)$. Let $\Delta\colon \cC(X)\to\cC(X)$ be the operator
\[
\Delta f(x)=\frac1{m(x)}\sum_{y\in X} b(x,y)\nabla_{xy}f, \quad f\in \cC(X),
\] 
where we interpret the right-hand side as divergence form operator. Note that $\Delta$ is non-negative and symmetric in $\ell^2(X,m)$. By abusing notation, we also denote by $\Delta\geq 0$ iits Friedrichs extension. We let
\[
\Lambda:=\inf\spec(\Delta)
\]
be the infimum of the spectrum of $\Delta$. The heat kernel $p$ is the minimal integral kernel of the heat semigroup $(P_t)_{t\geq 0}$, i.e., $P_t=\euler^{-t\Delta}$, $t\geq 0$, and for all $f\in \ell^2(X,m)$ and $t\geq 0$ we have
\[
P_t f(x)=\sum_{x\in X} m(y)p_t(x,y)f(y).
\]
Moreover, for any $f\in \ell^2(X,m)$, the map $t\mapsto P_tf$ solves the heat equation
\[
\frac{d}{dt} P_tf=-\Delta P_tf, \quad t\geq 0.
\]

A pseudo-metric $\rho\colon X\times X\to[0,\infty)$ is called intrinsic with respect to $b$ on $(X,m)$ if 
\[
\sum_{y\in X} b(x,y)\rho^2(x,y)\leq m(x), \quad x\in X.
\]
In the following, $\rho$ will always be a non-trivial intrinsic (pseudo-) metric with finite jump size
\[
S:=\sup\{\rho(x,y)\colon x,y\in X, b(x,y)>0\}>0.
\]

For $r\geq 0$ and $o\in X$, we denote 
$$B_o(r):=\{x\in X\colon \rho(x,o)\leq r\}.$$
We abbreviate $B(r)=B_o(r)$ if there is no confusion about the role of $o$.
A standing assumption on the intrinsic metric is the following.
\begin{assumption} The distance balls with respect to the intrinsic metric are compact and the jump size is finite.
\end{assumption}
{For $ x\in X $ and $ f\in\cC(X) $ we define 
	\[
	\vert \nabla f\vert(x):=\Bigg(\frac{1}{m(x)}\sum_{y\in X}b(x,y) (\nabla_{xy}f)^2\Bigg)^\frac{1}{2}.
	\]
	The combinatorial interior of a set $ A\subset X $
	will be denoted by 
	\begin{align*}
		A^{\circ}=\{x\in A\colon b(x,y)=0 \mbox{ for all }y\in X\setminus A \}.
	\end{align*}
}

The results presented in here are obtained by assuming the classical Sobolev and volume doubling assumptions. We encode them for weighted graph Laplacians in terms of intrinsic metrics. The big difference to the assumptions on manifolds is that we require these properties only on annuli with positive inradius instead of all small balls. 
\begin{definition}\label{def:sobvol} Let $b$ be a graph over $(X,m)$, $o\in X$, $R_2\geq R_1\geq 0$, $n>2$, and $d>0$.
\begin{enumerate}[(i)]
\item The Sobolev inequality $S(n,R_1,R_2)$ holds (in $o$), if there exists a constant $C_S>0$ such that for all $R\in[R_1,R_2]$, $u\in\cC(X)$, $\supp u\subset B(R)^{\circ}$, we have
\begin{equation*}
\frac{m(B(R))^{\frac{2}{n}}}{C_SR^2}
\Vert u\Vert_{\frac{2n}{n-2}}^2
\leq \Vert \vert\nabla u\vert\Vert_2^2+\frac{1}{R^2}\Vert u\Vert_2^2.
\end{equation*}
We abbreviate $S(n,R_1):=S(n,R_1,R_1)$.
\item The volume doubling property $V(d,R_1,R_2)=V(d,R_1,R_2)$ is satisfied (in $o$) if there exists $C_D>0$ such that 
\[
m(B(r_2))\leq C_D \left(\frac{r_2}{r_1}\right)^d m(B(r_1)),\quad R_1\leq r_1\leq r_2\leq R_2.
\]

\item 
The property $SV(R_1,R_2)$ respectively $SV(R_1,R_2,d,n)$  is satisfied (in $o$) if Sobolev $S(n,R_1,R_2)$ and volume doubling $V(d,R_1,R_2)$ hold.
\end{enumerate}
\end{definition}

\begin{remark}
We can replace $V(d,R_1,R_2)$ by the equivalent property $V^\ast(d,R_1,R_2)$
\[
m(B(2r))\leq C_D^\ast m(B(r)), \qquad r\in[R_1,R_2].
\]
$V(d,R_1,R_2)$ implies $V^\ast(d,R_1,R_2)$ and  $V^\ast(d,R_1,R_2)$ implies  $V(d,R_1,R_2/2)$, however,  with different constants.
Due to this asymmetry in the radii the volume doubling assumption $V(d,R_1,R_2)$ is more natural.
\end{remark}




%
%
%
%
%
%
%
%

%
%
%
%
%
%
%
%
%

%
%
%
%
%
%
%
%
%
%
%
%

%
%
%
%
%
%
%
%
%

\section{Anchored Gaussian upper bounds}\label{section:anchoredgaussian}

In order to formulate our main result, we recall the following weighted means of the degree and the inverse measure introduced in \cite{KellerRose-22a}. 
The weighted vertex degree is given, for $ x\in X $, by
\[
\Deg(x):=\frac{\deg(x)}{m(x)}=\frac{1}{m(x)}\sum_{y\in X} b(x,y).
\]
For $R\geq 0$, and $p\in(1,\infty)$, we define 
\begin{align*}
D_p(R)&:=D_p(o,R):=\left(\frac{1}{m(B(R))}\sum_{y\in B(R)}m(y)\Deg(y)^p\right)^{\frac{1}{p}},\quad
\\
 M_p(R)&:=M_p(o,R):=\left(\frac{1}{m(B(R))}\sum_{y\in B(R)}m(y)\frac{1}{m(y)^p}\right)^{\frac{1}{p}},
\end{align*}
and 
\[
D_\infty(R):=D_\infty(o,R):=\sup_{B(R)} \ \Deg \quad \text{and}\quad M_\infty(R):=M_\infty(o,R):=\sup_{B(R)}\ \frac{1}{m}.
\]

We set for $  \beta\in(1,1+1/q) $
\[
\theta(r):=\frac{1}{2\beta^{\kappa(r)}}, \qquad  \kappa(r):=\left\lfloor\sqrt{\frac{r}{4S}}-2\right\rfloor.
\]
Define the error-function $\Gamma(r)
:=
\Gamma_o(r,p,n,\beta)\geq 0$ by
\[
\Gamma(r)=\left[\left(1+r^2D_p\left(r\right)\right)
M_p\left(r\right)^q m\left(B\left(r\right)\right)^{q}\right]^{\theta(r)}.
\]

Our main theorem reads as follows.

\begin{theorem}\label{thm:main1}Let $d>0$, $n>2$, $p\in(1,\infty]$, $ \alpha=1+2/n $, $  \beta=1+2/(n\vee 2q)  $, and
\[R_2\geq  4 R_1\geq 32S\left(\frac{\ln q}{\ln\frac{\alpha}{\beta}}+3\right)^2.\]
Assume $SV(R_1,R_2)$. 
Then we have for all $x,y\in B(R_2/4)$ and $t\geq 2R_1^2$
\begin{align*}
p_t(x,y)
&
\leq 
C\ \Gamma(r(t,x))\Gamma(r(t,y))\left(1+\frac{\rho(o,x)^2+\rho(o,y)^2}{t\wedge R_2^2}\right)^{\frac{n}{2}}
\\
&\quad\cdot \frac{\left(1\vee (S^{-2} \sqrt{t^2+\rho(x,y)^2S^2}-t)\right)^{\frac{n}{2}}}{m(B(\sqrt t\wedge R_2))}\euler^{-\Lambda \left(t-t\wedge R_2^2\right)-\zeta_S(\rho(x,y),t)},
\end{align*}
where  $r(t,x):= \rho(o,x)+(\sqrt {t/2}\wedge R_2/4)$ and $C=2^{3n+2d+2}\euler C_{d,n,\beta}C_D$.
\end{theorem}

As initially observed by Davies, \cite{Davies-93}, instead of the Gaussian $\euler^{-r^2/4t}$ known from manifolds, for graphs the function $\euler^{-\zeta_S(r,t)}$ with 
  \[
\zeta_S(r,t):=\frac{1}{S^2}\left(r S \arsinh\left(\frac{r S}{t}\right)+t-\sqrt{t^2+r^2S^2}\right),
\]
for $ r\geq 0$, $t>0$, appears, where $ S $ is the jump size of the intrinsic metric.

In order to prove Theorem~\ref{thm:main1}, we use two results from \cite{KellerRose-22a} based on Davies' method. Here we combine these two propostions in a different way than in \cite{KellerRose-22a} to obtain anchored estimates. To obtain estimates  of solutions of the heat equation, we investigate properties of solutions of the $\omega$-heat equation
\[
\frac{d}{dt} v_t= -\Delta_\omega v_t,
\]
where $\Delta_\omega:= \euler^\omega\Delta \euler^{-\omega}$ is a sandwiched Laplacian for $\omega\in\ell^\infty(X)$. 
The following proposition is the first ingredient of the proof of the main theorem. It provides an $\ell^2$-mean value inequality for non-negative solutions of the $\omega$-heat equation. The displacement of the solutions with respect to the heat equation is measured in terms of the function
\[
h(\omega)=\sup_{x\in X}\frac{1}{m(x)}\sum_{y\in X} b(x,y)\vert \nabla_{xy}\euler^\omega\nabla_{xy}\euler^{-\omega}\vert.
\]
 
\begin{proposition}[{\cite[Theorem~4.2]{KellerRose-22a}}]\label{prop:l2meanvalue} Let $ x\in X $, $d>0$, $n>2$,  $p\in(1,\infty]
$, $\alpha=1+\tfrac2n$,
\[
\beta=1+\frac{1}{n\vee 2q}, 
\quad\mbox{
and
}\quad
R\geq 
8S\left(\frac{\ln q}{\ln\frac{\alpha}{\beta}}+3\right)^2.
\] 
There is a constant $C_{d,n,\beta}>0$ such that if $SV(R/2,R)$ holds in $o$,  then for $\tau\in(0,1]$, $T\in\RR$,  
 $\omega\in\ell^\infty(X)$, and all non-negative solutions $v\geq 0$ of the $\omega$-heat equation on $[T-R^2,T+R^2]\times B(R)$ we have 
\begin{align*}
\sup_{[T-\tau R^2/8,T+\tau R^2/8]\times B(R/2)}v^2 
\leq 
\frac{C_{d,n,\beta} \Gamma( R/2)^2}
{\tau^{\frac{n}{2}+1}R^{2}m(B(R))}(1+\tau R^2h(\omega))^{\frac{n}{2}+1}
\int\limits_{T-\tau R^2}^{T+ \tau R^2}\sum_{B(R)} m\ v_t^{2}\ \drm t.
\end{align*}
\end{proposition}

Proposition~\ref{prop:l2meanvalue} suffices to derive off-diagonal heat kernel bounds as shown by the following proposition, our second ingredient, in points where $\ell^2$-mean value inequalities are available. To this end, observe that, denoting $P_t^\omega:=\euler^\omega P_t\euler^{-\omega}$, the semigroup $(P_t^\omega)_{t\geq 0}$ acts on $\ell^2(X,m)$. Moreover, the map $t\mapsto P_t^\omega f$ solves the $\omega$-heat equation for $f\in\ell^2(X,m)$ and $\omega\in\ell^\infty(X)$.
\begin{proposition}[{\cite[Theorem~5.3]{KellerRose-22a}}]\label{theorem:daviesabstractgraph}Let $T>0$,
$Y\subset X$, and
 $a, b\colon Y \to [0,\infty)$, $a\leq b$, and $\phi\colon Y\times [0,\infty)\to [0,\infty)$ such that 
for all $f\in \ell^2(X), f\geq 0$, $\omega\in\ell^\infty(X)$, and $x\in Y$ we have
\[
\phi(x, h(\omega))^{2}(P_T^\omega f)^2(x)\leq  \int_{a(x)}^{b(x)}\Vert P_t^\omega f\Vert_2^2\ \drm t.
\]
Then we have 
for all $x,y\in Y$
\begin{align*}
p_{2T}(x,y)
&\leq 
\frac
{
(b(x)-a(x))^{\frac{1}{2}}(b(y)-a(y))^{\frac{1}{2}}
\exp\left(\frac{b(x)+b(y)-2T}{2}\sigma(\rho(x,y),2T)\right)
}
{\phi\big(x,\sigma(\rho(x,y),2T)\big)\phi\big(y,\sigma(\rho(x,y),2T)\big)}
\\
&\quad \cdot\exp\big(-\Lambda(a(x)+a(y))-\zeta_S(\rho(x,y),2T)\big),
\end{align*}
where
\[
\sigma(r,t):=2 S^{-2}\left(\sqrt{1+\frac{r^2S^2}{t^2}}-1\right).
\]
\end{proposition}

Combining these propositions in a suitable way proves Theorem~\ref{thm:main1}.
\begin{proof}[Proof of Theorem~\ref{thm:main1}]
$SV(R_1,R_2)$ yields $SV(r/2,r)$ for all radii $r\in[2R_1,R_2]$.
For $t\geq 2R_1^2$, $x\in B(R_2/4)$,  we set
\begin{align*}
R:= 2r(t,x)=2\rho(o,x)+2\left(\sqrt{\frac{t}{2}}\wedge \frac{R_2}{4}\right).
\end{align*}
Then,
\[
 R\in [2R_1,R_2]\quad \text{and}\quad x\in {B(R/2)}.
\]
We infer from
Proposition~\ref{prop:l2meanvalue} for all $t\geq 2R_1^2$, $x\in B(R_2/4)$, $\tau\in(0,1]$, $\omega\in\ell^\infty(X)$, and $f\in \ell^2(X,m)$, $f\geq 0$
\begin{align*}
(P_{\frac{t}{2}}^\omega f)^2(x)
&\leq \sup_{(s,y)\in [\frac{t}{2}-\tau {R^2}/{8},\frac{t}{2}+\tau {R^2}/{8}]\times B(R/2)}(P_s^\omega f)^2(y)
\\
&
\leq 
\frac{C_{d,n,\beta} \Gamma(R/2)^2}{\tau^{\frac{n}{2}+1}R^2m(B(R))}(1+\tau R^2h(\omega))^{\frac{n}{2}+1}
\int\limits_{\frac{t}{2}-\tau R^2}^{\frac{t}{2}+\tau R^2}\Vert P_s^\omega f\Vert_2^2\ \drm s.
\end{align*}
For $t\geq 2R_1^2$ and $x\in B(R_2/4)$ set
\[
a(x)=\frac{t}{2}-\tau (2r(t,x))^2,\quad b(x)=\frac{t}{2}+\tau (2r(t,x))^2,
\]
and
\begin{align*}
\phi(x)^{-1}&:=\sqrt{C_{d,n,\beta} }
\frac{\Gamma(r(t,x))}{\sqrt{\tau^{\frac{n}{2}+1}r(t,x)^2m(B(r(t,x)))}}(1+\tau (2r(t,x))^2h(\omega))^{\frac{n}{4}+\frac12}.
\end{align*}
Proposition~\ref{theorem:daviesabstractgraph} yields for any $t\geq 2 R_1^2$, $x,y\in B(o,R_{2}/4)$, and $\tau\in(0,1]$ 
\begin{align*}
&p_t(x,y)=p_{2(\frac{t}{2})}(x,y)\\
&\leq 
\frac
{
(b(x)-a(x))^{\frac{1}{2}}(b(y)-a(y))^{\frac{1}{2}}
\exp\left(\frac{b(x)+b(y)-t}{2}\sigma(\rho(x,y),t)\right)
}
{\phi\big(x,\sigma(\rho(x,y),t)\big)\phi\big(y,\sigma(\rho(x,y),t)\big)}
\\
&\quad \cdot\exp\big(-\Lambda(a(x)+a(y))-\zeta_S(\rho(x,y),t)\big)
\\
&
=
\frac{2C_{d,n,\beta}\Gamma(r(t,x))\Gamma(r(t,y))\tau^{-\frac{n}{2}}}{\sqrt{m(B(r(t,x)))m(B(r(t,y)))}}
\exp\left(\tau\frac{(2r(t,x))^2+(2r(t,y))^2}{2}\sigma(\rho(x,y),t)\right)
\\
&\quad\cdot(1+\tau (2r(t,x))^2\sigma(\rho(x,y),t))^{\frac{n}{4}+\frac{1}{2}}
(1+\tau (2r(t,y))^2\sigma(\rho(x,y),t))^{\frac{n}{4}+\frac{1}{2}}
\\
&\quad\cdot\exp\big(-\Lambda(t-\tau((2r(t,x))^2+(2r(t,y))^2))-\zeta_S(\rho(x,y),t)\big).
\end{align*}
We estimate the upper bound for the heat kernel further and start with the volume terms. Since $t\geq 2R_1^2$ and $4 R_1\leq R_2$ we have 
$
r(t,x)\geq \sqrt{\tfrac{t}{2}}\wedge \tfrac{R_2}{4}\geq R_1.
$
Thus, volume doubling implies
\[
\frac{1}{m(B(r(t,x)))}\leq \frac{1}{m(B(\sqrt{t/2}\wedge R_2/4))}
\leq 
\frac{C_D\left(\frac{\sqrt{t}\wedge R_2}{\sqrt{t/2}\wedge R_2/4}\right)^d}{m(B(\sqrt{t}\wedge R_2))}
\leq \frac{4^dC_D}{m(B(\sqrt{t}\wedge R_2))}.
\]

Now we choose $\tau\in(0,1]$ and estimate the remaining terms. 
We let
\[
\tau=1\wedge \left(\frac{t\wedge R_2^2}{(2r(t,x))^2+(2r(t,y))^2}\right)\wedge\left(\frac{1}{((2r(t,x))^2+(2r(t,y))^2)\sigma(\rho,t)}\right)
\]
satisfying $\tau\in(0,1]$
and immediately get 
\[
(1+\tau(2r(t,x))^2\sigma(\rho(x,y),t))^{\frac{n}{4}+\frac{1}{2}}\leq 2^{\frac{n}{4}+\frac{1}{2}},
\]
a similar estimate for $r(t,x)$ replaced by $r(t,y)$, and 
\[
\exp\left(\tau\frac{(2r(t,x))^2+(2r(t,y))^2}{2}\sigma(\rho(x,y),t)\right)
\leq \euler.
\]
Moreover, we obtain
\[
\exp\big(-\Lambda(t-\tau((2r(t,x))^2+(2r(t,y))^2))-\zeta_S(\rho(x,y),t)\big)\!
\leq\! \exp\big(-\Lambda(t-t\wedge R_2^2)-\zeta_S(\rho(x,y),t)\big).
\]

Now, we estimate $\tau^{-\frac{n}{2}}$. 
By the definition of $r(t,z)$, $z\in \{x,y\}$, we have
\[
(2r(t,z))^2\leq 8(\rho(o,z)^2+t\wedge R_2^2).
\]
Since \[\sigma(\rho,t)=\frac{2}{S^2}\left(\sqrt{1+\frac{\rho^2S^2}{t^2}}-1\right),\]
 we obtain
\begin{align*}
&\frac{1}{\tau}
\leq 
1\vee \left(\frac{(2r(t,x))^2+(2r(t,y))^2}{t\wedge R_2^2}\right)
\vee\left(((2r(t,x))^2+(2r(t,y))^2)\sigma(\rho,t)\right)
\\
&\leq 16\cdot 
1\vee\! \left(\frac{\rho(o,x)^2+\rho(o,y)^2}{t\wedge R_2^2}+1\right)
\vee\left(\!\left(\frac{\rho(o,x)^2+\rho(o,y)^2}{t}+1\right)\frac{2}{S^2}\left(\sqrt{t^2+\rho^2S^2}-t\right)\!\right)\!.
\end{align*}
Plugging in the above estimates into the upper bound for $p_t(x,y)$ yields the claim.
\end{proof}

	\section{Isoperimetric and Sobolev inequalities}\label{section:iso}
	
	In this section we show that isoperimetric inequalities formulated in terms of intrinsic metrics imply Sobolev inequalities. These results will be used in Section~\ref{section:antitrees} in order to obtain heat kernel upper bounds for anti-trees.
	
	For a graph $ b $ over $ (X,m) $, an intrinsic metric $ \rho $, $ U\subset X $ and $ n\in (2,\infty] $, let
	\begin{align*}
		h_{U,n}= \inf_{W\subseteq U \ \mbox{{\scriptsize finite}}}\frac{b\rho (\partial W)}{m(W)^{\frac{n-2}{n}}},
	\end{align*}
	where 
	\begin{align*}
		\partial W = \{(x,y)\in W\times(X\setminus W) \cup (X\setminus W) \times W\mid b(x,y)>0   \}.
	\end{align*}
Note that $h_{U,\infty}$ is also known as the Cheeger constant of $U$. It was shown in \cite{BauerKW-15} that the bottom of the spectrum can be bounded in terms of $h_{U,\infty}$. These considerations directly extend to bound the first Dirichlet eigenvalue of a finite set. In the following we adapt this proof in order to obtain a lower bound on the Sobolev constant instead. In order to do so, we collect some lemmas.

The following lemma is well-known, we included a proof in the Appendix \ref{appendix} for the reader's convenience.
	\begin{lemma}\label{lem:decreasingfunction}
		For a decreasing function $ F:[0,\infty)\to[0,\infty) $ and $ \alpha \in(0,1] $,
		\begin{align*}
			\left(\frac{1}{\alpha}\int_{0}^{\infty}F(t)t^{\frac{1}{\alpha}-1}\drm t\right)^{\alpha}\leq \int_{0}^{\infty}F(t)^{\alpha}\drm t.
		\end{align*}
	\end{lemma}
	
	As usual, we denote for a function $ f:X\to\RR $ and $ t\in\RR $
	\begin{align*}
		\{ f> t\}=\{ x\in X\mid f(x)>t \}.
	\end{align*}
	
	The following lemma is found in \cite[Lemma~10.8]{KellerLW-21}.
	\begin{lemma}[Co-area formula] For $ w\colon X\times X\to[0,\infty) $ and  $ f:X\to \RR $, we have 
		\begin{align*}
			\sum_{x,y\in X}w(x,y)\vert f(x)-f(y)\vert &=\int_{0}^{\infty}w(\partial \{f>t\} )\drm t.
		\end{align*}
	\end{lemma}

The next lemma is an adapted version of the area formula found in \cite[Lemma~10.9]{KellerLW-21}. This corresponds to the special case $\alpha=1$.
	\begin{lemma}[Area formula] For $ f:X\to [0,\infty)$ and $ 0<\alpha \leq  1$
		\begin{align*}
			\frac{1}{\alpha }	\sum_{x\in X}m(x) f(x)^{\frac{1}{\alpha}} &=\int_{0}^{\infty}m(\{f>t\} )t^{\frac{1}{\alpha}-1}\drm t.
		\end{align*}
	\end{lemma}
	\begin{proof}
		We calculate
		\begin{align*}
			\int_{0}^{\infty}t^{\frac{1}{\alpha}-1}m(\{f>t\} )\drm t&	=
			\int_{0}^{\infty}t^{\frac{1}{\alpha}-1}\sum_{x\in \{f>t\} }m(x)\drm t	=
			\int_{0}^{\infty}t^{\frac{1}{\alpha}-1}\sum_{x\in X }m(x)\Eins_{\{f>t\}}(f(x))\drm t\\
			&	=
			\sum_{x\in X} m(x)\int_{0}^{\infty} t^{\frac{1}{\alpha}-1}1_{\{f>t\}}(f(x))\drm t	=
			\sum_{x\in X}m(x) \int_{0}^{f(x)} t^{\frac{1}{\alpha}-1}\drm t\\
			&	= \frac{1}{\alpha}
			\sum_{x\in X} m(x){f(x)}^{\frac{1}{\alpha}}.
		\end{align*}
		This settles the claim.
	\end{proof}

	\begin{theorem}[Sobolev inequality and isoperimety]\label{t:sobolev} For $ n\in(2,\infty] $, $ \phi\in \cC_{c}(U) $ and $ C>0 $ we have
		\begin{align*}
			\frac{h_{U,n}}{C}
			\left(\frac{n}{n-2}\sum_{x\in X}m(x)\vert\phi{(x)}\vert^{\frac{2n}{n-2}}\right)^{\frac{n-2}{n}}\leq \sum_{x,y\in X}b(x,y)(\nabla_{xy}\phi)^{2}+\frac{1}{C^{2}}\sum_{x\in X}m(x)\phi(x)^{2}. 	 
		\end{align*}
	\end{theorem}
	\begin{proof} We employ the area formula with $ \alpha =(n-2)/n	 $, the basic inequality from Lemma~\ref{lem:decreasingfunction}, the co-area formula, the definition of $h_{U,n}$, and the fact that $\rho$ is intrinsic to obtain for $ \phi\in \cC_{c}(U) $
		\begin{align*}
		&	h_{U,n}	\left(\frac{n}{n-2}\sum_{x\in X}m(x)\vert\phi{(x)}\vert^{\frac{2n}{n-2}}\right)^{\frac{n-2}{n}}
			=
			h_{U,n}	\left(\frac{n}{n-2}\int_{0}^{\infty} m(\{\phi^{2}>t\})t^{\frac{n}{n-2}-1} \drm t\right)^{\frac{n-2}{n}}\\
			&\leq
			h_{U,n}\int_{0}^{\infty} m(\{\phi^{2}>t\})^{\frac{n-2}{n}} \drm t
			\leq \int_{0}^{\infty} b\rho (\partial \{\phi^{2}>t\}) \drm t
			=
			\sum_{x,y\in X}b \rho(x,y)\vert  \phi^{2}(x)-\phi^{2}(y)\vert \\
			&\leq \frac{1}{2}\left(2C\sum_{x,y\in X}b(x,y)( \phi(x)-\phi(y))^{2}+
			\frac{1}{2C}\sum_{x,y\in X}b(x,y)\rho(x,y)^2(\phi(x)+
			\phi(y))^{2}\right)\\
			&\leq
			{C}\sum_{x,y\in X}b(x,y)( \phi(x)-\phi(y))^{2}+\frac{1}{C}\sum_{x\in X}m(x)\phi(x)^{2}. 	 
		\end{align*}
		This proves the claim.
	\end{proof}
	
	\section{Anti-trees}\label{section:antitrees}
	In this section we show that anti-trees of a certain growth fulfill the assumptions of Theorem~\ref{thm:main1}. This is can be achieved by reducing the analysis on the anti-tree to a one-dimensional problem. 
	
	We consider here graphs $ b $ with standard weights that is $ b(x,y)\in\{0,1\} $, $ x,y\in X $.  For a connected graph $ b $ over $ X $, we denote  by $$  S_{k}=\{x\in X\mid d(x,o)=k\} ,\qquad k\in\mathbb{Z}, $$ the distance spheres with respect to the combinatorial graph distance $ d $ to the root $ o\in X $. Clearly, $ S_{-k}=\emptyset $ for $ k\in \NN $. We furthermore denote
	\begin{align*}
		|x|=d(x,o),\qquad x\in X.
	\end{align*}
	
	For a so-called sphere function $ s:\NN  \longrightarrow \NN$ the corresponding anti-tree is a graph $ b $ with standard weights over a set $X$
	such that for the distance spheres we have 
	$$  \#S_{k}=s_{k}  $$ for $ k\ge 1$,
	\begin{align*}
		b(x,y)=b(y,x)=1,\qquad x\in S_{k},y\in S_{k+1}
	\end{align*}
	for $ k\ge0  $ and $ b(x,y)=0 $ otherwise. 
	For convenience later on, we set $ s_{0}=1 $ and $ s_{-1}=0 $. 
	Moreover, we equip $ X $ with the counting measure $ m=1 $ and denote for $ k\in\mathbb{Z} $ the combinatorial balls by
	\begin{align*} 		A_{k}=S_{0}\cup\ldots\cup S_{k}. 	
	\end{align*}

	\begin{lemma}[Characterization anti-trees] Let $ b $ be connected  a graph with standard weights such that there are no edges within spheres. The following statements are equivalent:
		\begin{itemize}
			\item [(i)] $ b  $ is an anti-tree.
			\item [(ii)]  For any  vertices $ x,y,x',y' $ such that $ |x|=|x'| $ and $ |y|=|y'| $, there is a graph isomorphism which interchanges $ x $ and $ x' $ as well as $ y $ and $ y' $.
		\end{itemize}	
	\end{lemma}
	\begin{proof}
		(i) $ \Longrightarrow $ (ii): For an anti-tree the  map which interchanges $ x $ with $ x' $ and $ y $ with $ y' $ and leaves every other vertex invariant is a graph isomorphism since $ x $ and $ x' $ as well as $ y $ and $ y' $ have the same neighbors each.
		
		(ii) $ \Longrightarrow $ (i): The forward graph $ F_{x} $ of a vertex $ x $ with respect to a root $o  $ is a subgraph of $ b $ induced by the vertices $ z\in \bigcup_{k\ge |x|}S_{k} $ such that $ d(x,z)=k $ for $ z\in S_{k} $. 
		
		Clearly, a graph isomorphism which interchanges $ x $ and $ x' $ with $ |x|=|x'| $ also has to map the forward graph $ F_{x} $ to $ F_{x'} $ and vice versa. 		
		If $ b $ is not an anti-tree, then there is a vertex $ x $ which is not connected to some $ y\in S_{|x|+1} $, i.e., $ y\not\in F_{x} $. Furthermore, there is $ x' \in S_{|x|}$ which is connected to $ y $, i.e., $ y\in S_{|x'|+1} $. Thus, there is no graph isomorphism interchanging $ x $ and $ x' $ and which keeps $ y=y' $ invariant:  if there was such a $ \iota $, then $$ 0=b(x,y)= b(\iota(x'),\iota(y))=b(x',y) =1$$
		which is a contradiction.
	\end{proof}

	A function $ f:X\longrightarrow\RR $ is called \emph{spherically symmetric}, if there is a function $ \tilde f:\NN_{0}\longrightarrow \RR $ such that
	\begin{align*}
		f(x)=\tilde f(|x|),\qquad x\in X.
	\end{align*}

	For a spherically symmetric function $ f $, the corresponding Laplacian acts as
	\begin{align*}
		\Delta f(x)= s_{|x|-1}(\tilde f(|x|)-\tilde f(|x|-1))+s_{|x|+1}(\tilde f(|x|)-\tilde f(|x|+1)).
	\end{align*}
	
	To a given anti-tree with sphere function $ s $, we can associate a one-dimensional weighted graph $ \tilde  b$ over $ (\tilde X,\tilde m) $
	with \begin{align*}
		\tilde X&=\NN_{0},\\
		\tilde m(k)&=m(S_{k})=s_{k},\\
		\tilde b(k,k+1)&=\tilde b(k+1,k)=s_{k}s_{k+1}
	\end{align*}
	for $ k\ge0 $ and $\tilde b(k,l)=0 $ otherwise. For the corresponding Laplacian $ \tilde \Delta $ and a spherically symmetric function $ f $ and $ x\in X $ we have
	\begin{align*}
		\Delta f(x)=\tilde \Delta \tilde f(|x|).
	\end{align*}

	\begin{lemma}\label{l:1d}
		For an anti-tree $ b $ over $ (X,m) $ with heat kernel $ p $ and corresponding graph $ \tilde b  $ over $ (\NN_{0},\tilde m) $ with heat kernel $ \tilde p $, we have
		\begin{align*}
			p_{t}(x,y)=\tilde p_{t}(|x|,|y|),\qquad  t\ge0,
		\end{align*}
 for 	$ x,y\in X $  with either $ |x|\neq |y| $ or $ x=y $.
	\end{lemma}
	\begin{proof}
		By the characterization of anti-trees for any $ x,y,x',y' $ such that $ |x|=|x'| $ and $ |y|=|y'| $ (and $ y=y' $ if and only if $ x=x' $), there is a graph isomorphism $ \iota $ interchanging $ x $ and $ x' $ as well as  $ y $ and $ y' $. This means that $ \iota $ extends to a unitary operator on $\ell^{2}(X,m)  $ via $ f\mapsto f\circ \iota $ which commutes with the Laplacian. Thus, it commutes with the semigroup. Hence, 	$ p_{t}(x,y)=	p_{t}(x',y') $ for all such $ x,y,x',y' $. In particular, the heat kernel  is spherically symmetric about $o$ in both variables, so there is a function $ q $ such that $ p_{t}(x,y)=q_{t}(|x|,|y|) $, $ x,y\in X $ with $ x=y $ if $ |x|=|y| $, $ t\ge0 $. Specifically, $ q $ solves the heat equation for $ \tilde b $, i.e.,
		\begin{align*}
			\tilde \Delta q_{t}(|x|,|y|)= \Delta p_{t}(x,y)=\partial_{t} p_{t}(x,y)=\partial_{t} q_{t}(|x|,|y|)
		\end{align*}
		and $ q_{0}(|x|,|y|)=\Eins_{|y|}(|x|) $. Since $(t,x)\mapsto p_{t} (x,y)$ is the smallest positive solution to the heat equation for $ b $, so is $ q $ for $ \tilde b $ and, hence, $ q $ is the heat kernel for $ \tilde b $.
	\end{proof}

	From now on, we will choose the intrinsic metric $ \rho  $ to be the intrinsic degree metric. Recall that for a graph $ b $ over $ (X,m) $ the intrinsic degree metric  is given by
	\begin{align*}
		\rho(x,y)=\inf_{x=x_{0}\sim \ldots\sim x_{k}=y}\sum_{j=0}^{k-1}\left(\frac{1}{\Deg(x_{j})}\wedge \frac{1}{\Deg(x_{j+1})}\right)^{1/2}
	\end{align*}
	for $ x,y\in X $,
	where for antitrees the weighted degree is given by
	\begin{align*}
		\Deg(x)=\frac{1}{m(x)}\sum_{y\in X}b(x,y)=s_{|x|-1}+s_{|x|+1}
	\end{align*}
	for $ x  \in X$. 
	
 For an antitree $ b $ and the associated one-dimensional graph $ \tilde b$, we denote the corresponding intrinsic degree metric $ \tilde \rho $ for $\tilde b  $ over $ (\NN_{0}, \tilde m) $.
	
	\begin{lemma}[Intrinsic metric]\label{l:intrinsic}	For $ x,y  \in X $ such that  either $ |x|\neq |y| $ or $ x=y $,
		\begin{align*}
			\rho(x,y)=\tilde \rho(|x|,|y|).
		\end{align*}
	\end{lemma} 
\begin{proof} 
		Observe that for the weighted degree 
$ \tilde{\Deg} $ of 
$\tilde b $	 
		we have 
	\begin{align*}
		\tilde{\Deg}(|x|)=\frac{1}{s_{|x|}}(s_{|x|}s_{|x|+1}+ s_{|x|-1}s_{|x|}) =
		(s_{|x|+1}+ s_{|x|-1})=\Deg(x). 
	\end{align*}
Moreover, for any path between $ x $ and $  y $ with either $ |x|\neq |y | $ or $ x=y $ there is a unique path on $ \NN_{0} $. On the other hand, all the  different paths in $ X $ corresponding to the same path on $ \NN_{0} $ have the same length. This settles the claim.
\end{proof}
	We denote the closed distance balls about $ x\in X $ and $ |x|\in \tilde X $ with respect to $ \rho $ and $ \tilde \rho $ by $ B_{x} (r)$ and $ \tilde B_{k}(x) $ for $ r\in\RR $. Moreover, we denote  the open balls in $X$ and $\tilde X$ by $ B_{x}^{\circ}(r) =\{y\in X\mid\rho(x,y)<r\}$  and  $ \tilde B_{|x|}^{\circ}(r) =\{|y|\in \tilde X\mid\tilde\rho(|x|,|y|)<r\} $ respectively.\medskip

	From now on for given  $ \gamma\in [0,2) $,  
	we consider the anti-tree with the sphere function 
	\begin{align*}
		s_{k-1}=[k^{\gamma}]
	\end{align*}
	$ k\ge1 $, where $ [\cdot] $ is the integer function.\medskip

The intrinsic and combinatorial distance can be estimated against each other which is elaborated in  \cite{Huang}. For convenience, we recall the argument here. For two real valued functions $ f $ and $ g $, we denote $$  f\asymp g  $$ if there is a constant $ C>0 $ such that $ C^{-1}f\leq g\leq Cf  $.
		\begin{lemma}[Distance]\label{distance}
		For $ \gamma\in[0,2) $, the corresponding anti-tree  satisfies for $ x,y\in X $ with either  $ |x|\neq |y| $ or $ x=y $
		\begin{align*}
			\rho(x,y)\asymp ||x|-|y||^{\frac{2-\gamma}{2}}.
		\end{align*} In particular, there are $ c,C\ge0 $ such for every $ r $ there is $ c\leq \theta \le C $ such that
		\begin{align*}
			B_{o}(r)=A_{\theta r^{\frac{2}{2- \gamma}} }.
		\end{align*}
	\end{lemma}
	\begin{proof}
		We have  $$  \mathrm{Deg}(z)=s_{|z|-1}+s_{|z|+1} \asymp |z|^{\gamma} $$ for $ z\in X $.  Then, for $ |x|>|y| $
		$$  		\rho(x,y)\asymp\sum_{j=|y|}^{|x|-1}j^{-\frac{\gamma}{2}}\asymp \int_{|y|+1}^{|x|}j^{-\frac{\gamma}{2}}dj \asymp (|x|-|y|)^{1-\frac{\gamma}{2}}.
		$$
		This proves the first statement. The "in particular" statement is a direct consequence.
	\end{proof}
	
	The following constant can be understood as the dimension of the antitree with parameter $ \gamma\in (0,2) $.
	\begin{align*}
		  d=\frac{2(\gamma+1)}{2-\gamma}.
	\end{align*}

	\begin{lemma}[Polynomial volume growth]\label{l:volume}
		For $ \gamma \in [0,2) $, $ r\ge0 $, $ x\in X $ and $ r_{x}=\rho(x,o) $, we have
		\begin{align*}
			\# B_{x}({r})=m(B_{x}({r}))=\tilde m(\tilde B_{x}({r}))\asymp
			\begin{cases}
				r^d&:r\ge r_{x} ,\\
				 rr_x^{d-1}&: r\leq r_{x}.
			\end{cases}
		\end{align*}
	In particular, for we have $ V(d,R,\infty) $ in every $ x\in X $ with $ R\ge \rho(x,o) $.
	\end{lemma}
	\begin{proof}The  equalities follow directly from the  definitions. Since for $ N\in \NN_{0} $,
		\begin{align*}
			\# A_{N}=\sum_{k=0}^{N }\#S_{k}=\sum_{k=0}^{N }s_{k}\asymp
			\sum_{k=0}^{N } k^{\gamma}\asymp\int_{0}^{N}k^{\gamma}dk\asymp N^{\gamma+1}.
		\end{align*}
	Thus, by the lemma above
	\begin{align*}
			m(B_{o}(r))\asymp r^{d}.
	\end{align*}
	Now, observe that  $ m(B_{o}^{\circ}(r_{x}-r))\asymp m(B_{o}(r_{x}-r)) $ and thus
	\begin{align*}
		m(B_{x}(r))\asymp m(B_{o}(r_{x}+r)) - m(B_{o}(r_{x}-r)) \asymp (r_{x}+r)^{d}-(r_{x}-r)_{+}^{d}.
	\end{align*}
	The statement for $ r\ge r_{x} $ is clear and the statement for $ r\leq r_{x} $ follows by the mean value theorem after factoring out $ r  $ on the right hand side.
	
	The claim about volume doubling follows since $ m(B_{r}(x))=m(B_{r+r_{0}}(o)) $ for all $ r\ge r_{0}=\rho(x,o) $.
	\end{proof}
	
	Next, we estimate the isoperimetric constant of distance balls which we need to obtain a Sobolev inequality via Theorem~\ref{t:sobolev}. For the one dimensional graph $ \tilde b $ associated to an antitree and a finite set $  B\subseteq \tilde X $ recall that  $$  \tilde h_{ B,n}=\inf_{W\subseteq B}  \frac{\tilde b\tilde \rho (\partial W) }{\tilde m (W)^{\frac{n-2}{2}}}. $$

	\begin{lemma}[Isoperimetic inequality]\label{lem:isoanti}
		For $ \gamma \in [0,2) $,  $ n\ge 2d=\frac{4(\gamma+1)}{2-\gamma}\ge 2 $,  $x\in X $ and $ r\ge 1\vee \rho(o,x)$, we have
			\begin{align*}
			{\tilde h_{\tilde B_{r}(|x|),n}}\asymp  \frac{\tilde m(\tilde B_{r}(|x|))^{\frac{2 }{n}}}{r}.
		\end{align*}

	\end{lemma}
	\begin{proof}Since $ r\ge 1\vee \rho(o,x)=:r_{0} $, we have  that $o \in  B_{r}(x) $. So,  to determine the infimum in $ \tilde h_{\tilde B_{r}(|x|),n} $, it suffices to consider balls $B_{R}= B_{R} (o)$ with $ R\le r_{0}+r $ and estimate
		\begin{align*}
			\frac{\tilde b\tilde \rho(\tilde B_{R})}{\tilde m(\tilde B_{R})^{\frac{n-2}{2}}}= 	\frac{ b \rho( B_{R})}{ m( B_{R})^{\frac{n-2}{2}}}
		.
		\end{align*}
		 For $ R\ge0 $, we recall the bound $ m(B_{R})\asymp R^{(2(\gamma+1))/(2-\gamma)}  $ from Lemma~\ref{l:volume}. 		
		Since $ 	\Deg(x)=s_{|x|-1}+s_{|x|+1} $ we have
		$ \rho_{\vert x\vert}:=	\rho(x,y)=(s_{|x|}+s_{|x|+2})^{-\frac{1}{2}} $ for $ y\in S_{|x|+1} $. 		Moreover, also by Lemma~\ref{l:volume}, for every $ R\ge0 $ there  is $c\le \theta\le C $ such that $ 	B_{R}=A_{k} $ with $ k=\theta R^{ {2}/({2- \gamma})}  $  and since  $ s_{k}\asymp k^{\gamma} $, we have
		\begin{align*}
			b\rho(\partial B_{R})=\rho_{k}\#\partial A_{k} =(s_{k}+s_{k+2})^{-\frac{1}{2}}s_{k} s_{k+1}\asymp {k}^{\frac{3}{2}\gamma} \asymp R^{ \frac{3\gamma }{2- \gamma}}.
		\end{align*}
		Now,
		\begin{align*}
			\frac{3\gamma }{(2- \gamma)} -\frac{2(\gamma+1)}{(2-\gamma)}\frac{{(n-2)}}{n}=\frac{\left( 3\gamma n - 2(\gamma+1)(n-2)\right)}{n(2-\gamma)}= -1 +\frac{2}{n}\cdot\frac{2(\gamma+1)}{(2-\gamma)}.
		\end{align*}
	 Since $ \frac{4(\gamma+1)}{n(2-\gamma)}\leq 1 $, we have
	\begin{align*}
		{\tilde h_{\tilde B_{r}(|x|),n}} =\inf_{R\le r_{0}+r}\frac{	\tilde b\tilde \rho(\partial \tilde B_{R})}{\tilde m(\tilde B_{R})^{\frac{{n-2}}{n}}}=
		\inf_{R\le r_{0}+r} R^{-1 +\frac{2}{n}\cdot\frac{2(\gamma+1)}{(2-\gamma)}}\asymp r ^{-1 +\frac{2}{n}\cdot\frac{2(\gamma+1)}{(2-\gamma)}}\asymp  \frac{\tilde m(\tilde B_{r}(|x|))^{\frac{2 }{n}}}{r}
		,
	\end{align*}
where the last estimate follows as $\tilde B_{r}(|x|)= \tilde B_{r_{0}+r}(0) $ for $ r\ge r_{0} $, i.e.~we have the bounds $m (\tilde B_{r}(|x|))\asymp r^{2(\gamma+1)/(2-\gamma)} $.
	This finishes the proof.
	\end{proof}

	We apply the estimates to obtain Sobolev inequalities by Theorem~\ref{t:sobolev}.
	
	\begin{proposition}[Sobolev inequality]\label{p:sobolev} For $ \gamma\in (0,2) $ and  $ n\ge 2d=\frac{4(\gamma+1)}{2-\gamma}\ge 2 $ there is $ C_{S}>0 $ such that for all $ r>0 $ and 
		all  functions $ \phi\in C_{c}(\tilde B_{r}) $
		\begin{align*}
			\frac{\tilde m(\tilde B_{r})^{\frac{2}{n}}}{C_{S}r^{2}}
			\left(\sum_{x\in \tilde  X}\tilde  m(x)\vert\phi{(x)}\vert^{\frac{2n}{n-2}}\right)^{\frac{n-2}{n}}\leq \sum_{x,y\in \tilde X} \tilde b(x,y)(\nabla_{x,y}\phi)^{2}+\frac{1}{r^{2}}\sum_{x\in \tilde X}\tilde  m(x)\phi(x)^{2} . 	 
		\end{align*}
	\end{proposition}
	\begin{proof}
	To apply Theorem~\ref{t:sobolev}, we choose $ C=r $ and conclude  the statement  from  Lemma~\ref{lem:isoanti}.
	\end{proof}

Next we estimate the averaged degree functions
	\begin{align*}
		D_{p}(x,r)=\left(\frac{1}{m(B_{r}(x))}\sum_{ B_{r}(x)}m\mathrm{Deg}^{p}\right)^{\frac{1}{p}}\quad \tilde D_{p}(\vert x\vert,r)=\left(\frac{1}{\tilde m(\tilde B_{r}(|x|))}\sum_{\tilde B_{r}(|x|)}\tilde m{\tilde{\mathrm{Deg}}}^{p}\right)^{\frac{1}{p}},
	\end{align*}
	for $ x \in X$ and $r\geq 0$.

	\begin{lemma}\label{l:D}
		For $ \gamma\in (0,2) $ and $ x\in X $, we have for all $ r\ge \rho(x,o) $ and $ p\in[1,\infty] $
		\begin{align*}
			D_{p}(x,r)=	\tilde D_{p}(|x|,r)\asymp  r^{\frac{2\gamma}{(2-\gamma)}}.
		\end{align*}
	\end{lemma}
	\begin{proof}
		The equality follows by definition.
		Recall that
		\begin{align*}
			s_{k}\asymp k^{\gamma}\quad \mbox{and}\quad \mathrm{Deg}(y)=(s_{k-1}+s_{k+1})\asymp k^{\gamma}
		\end{align*}
		for $ y\in S_{k} $ and $ k\ge0 $. Let $ p\in[1,\infty) $ and $ r_{0} =\rho(x,o)$. For  every $ r\ge r_0 $ there  is $c\le \theta\le C $ such that $ 	B_{r}(x)=B_{r+r_{0}}(o)=A_{R} $ with $ R=\theta (r+r_{0})^{ {2}/({2- \gamma})}  $. 		Thus, we have since $ m(B_{r}(x_{0}))= m(B_{r+r_{0}}(o))=  \tilde m(\tilde B_{r})\asymp (r+r_{0})^{\frac{2(\gamma+1)}{2-\gamma}} $
		\begin{align*}
			D_{p}(x,r)^{p}&=\frac{1}{m(B_{r+r_{0}}(o))}\sum_{k=1}^{R}\sum_{y\in S_{k}}\mathrm{Deg}(y)^{p}=\frac{1}{m(B_{r+r_{0}}(o))}\sum_{k=1}^{R}s_{k}(s_{k-1}+s_{k+1})^{p}\\
			&\asymp  (r+r_{0})^{-\frac{2(\gamma+1)}{2-\gamma}}  \int_{0}^{R}k^{(p+1)\gamma} dk \asymp   (r+r_{0})^{ \frac{2((p+1)\gamma+1)}{({2- \gamma})}-\frac{2(\gamma+1)}{2-\gamma}}			 \asymp r^{\frac{2p\gamma}{2-\gamma}}.
		\end{align*}
		This proves the asymptotics for $ p<\infty $.  For $ p=\infty $, one clearly obtains the bounds  $  D_{\infty}(x,r)=s_{\theta (r+r_{0})^{2/(2-\gamma)}}\asymp r^{2\gamma/(2-\gamma)}  $.
	\end{proof}
	
\eat{\begin{lemma}
	For $ \gamma\in (0,2) $, $ p\in (1,\infty)  $ with H\"older dual $ q $ and $ r_{x}=\rho(x,o) $ we have
	\begin{align*}
			\mu_x(r):=\left(\frac{m(B_x(r))}{r^d}\right)^q\asymp  \left(\frac{r_{x}}{r}\right)^{q(d-1)}.
	\end{align*}
\end{lemma}	
\begin{proof}
This follows directly from Lemma~\ref{l:volume}.
\end{proof}}
	
	Now, we collected all the ingredients to prove heat kernel estimates on anti-trees.
	
\begin{theorem}\label{thm:antitree1} For $ \gamma\in (0,2) $ and  $ n\ge 2d=\frac{4(\gamma+1)}{2-\gamma}> 2 $, there exists a constant $C>0$ such that for all $x,y\in X$ with either $ |x|\neq |y| $ or $ x=y $ and $t\geq 2\cdot 72^2$ we have  
	\begin{multline*}
		p_{t}(x,y)
		\leq 
		C
		\left(1+\frac{\rho(o,x)^2+\rho(o,y)^2}{t}\right)^{\frac{n}{2}}
				\cdot
		\frac{\left(1\vee (\sqrt{t^2+\rho(x,y)^2}-t)\right)^{\frac{n}{2}}}{m(B_{o}(\sqrt t))}\euler^{-\zeta_1(\rho(x,y),t)}.
	\end{multline*}
and with respect to the combinatorial distance for $t>2\max\{||x|-|y||^{2-\gamma},||x|-|y||^{\frac{2-\gamma}2}\}$
\begin{align*}
			p_{t}(x,y)
	\leq 
	C 
	\left(1+\frac{ |x|^{2({\gamma+1})}+|y|^{2(\gamma+1)}}{t^{d}}\right)
	\cdot \frac{ 1}{t^{\frac{d}{2}}}
	\euler^{-\frac{ ||x|-|y||^{2-\gamma}}{Ct}}.
\end{align*}
\end{theorem}
\begin{proof}By Lemma~\ref{l:intrinsic} and Lemma~\ref{l:1d}, it suffices to prove the statement for the one-dimensional kernel $ \tilde p $ with intrinsic metric $ \tilde \rho $.
	
 We check the assumptions of Theorem~\ref{thm:main1}. First of all, since $ \tilde m(k)=s_{k}\ge 1 $, we have $ \tilde M_{p}(R)\leq 1  $ for all $ p \in (1,\infty)$ and $ R \ge0 $. Furthermore, $ \tilde D_{p}(R) $ is polynomially bounded by Lemma~\ref{l:D}. Volume doubling $ VD(d,1,\infty) $ is satisfied by Lemma~\ref{l:volume} with $ d=2(\gamma+1)/(2-\gamma) $. Moreover, the Sobolev inequality $ S(n,1,\infty) $ is satisfied by Proposition~\ref{p:sobolev}. Hence, we have $ SV (1,\infty)$. 
 
 Now, the first result follows from  Theorem~\ref{thm:main1} since the jump size of $ \rho $ can be estimated by $ S\le 1 $.
 
 For the second statement with respect to the combinatorial graph metric, we choose $ n=2d $.
 For the first term on the right hand side, we use $ (1+\alpha)^{d}\leq 2^{d-1}(1+\alpha^{d}) $ and  we estimate the enumerator of the term in the middle by $ 1 $. Furthermore, we use $ \zeta (r,t)\geq \frac{r^{2}}{C t} $ for $t>c\rho(x,y) $, cf.~\cite[p.~214]{Delmotte-99}, and employ Lemma~\ref{distance} to estimate $ \rho  $ by $ |\cdot|^{(2-\gamma)/2} $.
\end{proof}

\begin{theorem}\label{thm:antitree2}Let $ \gamma\in (0,2) $, and $ n\ge 2d \frac{4(\gamma+1)}{2-\gamma}> 2 $. There exists $C_0,R_0>0$  such that  for all $x,y\in X$ with either $ |x|\neq |y| $ or $ x=y $ and  $t\geq 8\max\{\rho(x,o),\rho(y,o) ,R_{0}\}^{2}$ we have
		\begin{align*}
		p_{t}(x,y)
		&\leq 
	 C 
		\frac{ \left(1\vee\left(\sqrt{t^2+\rho(x,y)^2}-t\right)\right)^{\frac{n}{2}}}
		{\sqrt{m(B_{x}(\sqrt {t}))m(B_{y}(\sqrt {t}))}}
		\euler^{-\zeta_1\left(\rho(x,y),t\right)}.
	\end{align*}
\end{theorem}
\begin{proof}
Again 	by  Lemma~\ref{l:intrinsic}  and Lemma~\ref{l:1d}, it suffices to prove the statement for the one-dimensional kernel $ \tilde p $ and $ \tilde \rho $.
	
We apply \cite[Theorem~6.1]{KellerRose-22a}. We choose $\beta=1+1/n$, $ p=\infty $ and $ R_{0}=72 $. Then, for $x,y $ we have $ SV(r,\infty) $ for $ r\ge r_{x}\vee r_{y}\vee R_{0}  $ by Lemma~\ref{l:volume} and Proposition~\ref{p:sobolev}. 

The error terms $$  \Gamma_{x}(r)=[(1+r^2\tilde D_\infty(r))
\tilde M_\infty(r)\tilde m( \tilde B(r))]^{\theta(r)}  $$ 
are bounded since $\tilde D_{\infty}(r)$ , $ \tilde M_{\infty} (r)$, $m( \tilde B(r))$ are polynomially growing and we have  $\theta(r)\asymp \euler^{-\gamma\sqrt{r}}$. The polynomial volume growth yields $ \Lambda=0 $, cf.~\cite{HaeselerKW-13}. Moreover, the jump size can be bounded by $ S\leq 1 $. Hence, the result can be read from \cite[Theorem~6.1]{KellerRose-22a}.	
\end{proof}

{\noindent\textbf{Acknowledgements.} The authors acknowledge the financial support of the DFG. }
\appendix
\section{Appendix}\label{appendix}

\begin{proof}[Proof of Lemma~\ref{lem:decreasingfunction}]Let 
		\[
		g(T):=\frac{1}{\alpha}\int_{0}^{T}F(t)t^{\frac{1}{\alpha}-1}dt,\quad h(T):=\left(\int_{0}^{T}F(t)^{\alpha}dt\right)^{\frac{1}{\alpha}}.
		\]
		Then $g(0)=h(0)=0$, $g$ and $h$ are differentiable, and
		\[
		g'(T)=\frac{1}{\alpha} F(T)T^{\frac{1}{\alpha}-1}, \quad h'(T)=\frac{1}{\alpha}\left(\int_{0}^{T}F(t)^{\alpha}dt\right)^{\frac{1}{\alpha}-1}F(T)^\alpha.
		\]
		Since $F$ is non-increasing, we get
		\[
		h'(T)\geq \frac{1}{\alpha}F(T)^{\alpha\left(\frac{1}{\alpha}-1\right)}T^{\frac{1}{\alpha}-1}F(T)^\alpha=g'(T).
		\]
		Hence, $h(T)\geq g(T)$. Letting $T\to\infty$ yields the claim.
	\end{proof}

\bibliographystyle{alpha}

\end{document}